\documentclass[11pt,oneside]{amsart}

\usepackage{lmodern}    
\usepackage{mathpazo}
\usepackage{amsmath,amssymb,amsthm}
\usepackage{graphicx}
\usepackage[margin=1in]{geometry}
\usepackage[colorlinks=true]{hyperref}
        \hypersetup{urlcolor=RoyalBlue,linkcolor=RoyalBlue,citecolor=ForestGreen}
\usepackage[nameinlink]{cleveref}

\usepackage[usenames,svgnames,dvipsnames]{xcolor}
\theoremstyle{plain}
\newenvironment{subproof}[1][\proofname]{%
  \begin{proof}[#1]%
}{%
  \end{proof}%
}

\usepackage{enumitem}
\setenumerate[0]{label=(\alph*)}

\newcommand\ol\overline
\newcommand\dang\measuredangle

\newtheorem{theorem}{Theorem}[section]
\newtheorem{lemma}[theorem]{Lemma}
\newtheorem{corollary}[theorem]{Corollary}

\theoremstyle{definition}
\newtheorem{definition}[theorem]{Definition}

\theoremstyle{remark}

\numberwithin{equation}{section}

\begin{document}

\title{Constructions in the Locus of Isogonal Conjugates in a Quadrilateral}

\author{Daniel Hu}
\address{Los Altos High School, Los Altos, CA}
\email{daniel.b.hu@gmail.com}

\maketitle

\begin{abstract}
    Given fixed distinct points $A, B, C, D$, we examine properties of the locus of points $X$ for which $(XA, XC)$, $(XB, XD)$ are isogonal. This locus is a cubic curve circumscribing $ABCD$. We characterize all possible such cubics $\mathcal C \in \mathbb R^2$. These properties allow us to present constructions involving these cubics, such as intersections and tangent lines, using straightedge and compass.
\end{abstract}

\section{Introduction}

In this paper, we characterize the locus of all points $P$ with an isogonal conjugate in a given quadrilateral $ABCD$. This turns out to be a cubic plane curve, which we will call the \textit{isogonal cubic} of $ABCD$. The isogonal cubic is a well-established figure in geometry. However, its properties are often considered with respect to the base quadrilateral $ABCD$, without considering the isogonal cubic as an individual curve, and often neglecting degenerate cases of $ABCD$.

The first half of the paper is dedicated to discovering geometric properties of isogonal cubics, and also providing constructions on the isogonal cubic with straightedge and compass. This sets up the second half, which characterizes all possible non-degenerate cubics $\mathcal C \in \mathbb{RP}^2$ such that there exist $A, B, C, D \in \mathbb R^2$ for which $\mathcal C$ is the isogonal cubic of $ABCD$. We also establish the notion of the spiral center and isogonal conjugation purely with respect to a valid cubic $\mathcal C$. This formalizes constructions on the isogonal cubic with straightedge and compass, only requiring the cubic's unique spiral center and asymptote, without the base quadrilateral $ABCD$. In particular, these constructions do not require intersecting lines with cubic curves. This makes them compatible with software such as Geogebra, where intersecting lines with cubics is not always supported.

The following is the main result we prove in this paper, which underlies these constructions:

\begin{theorem}
    Let $\mathcal C$ be a non-degenerate cubic in $\mathbb R^2$, and let $\mathcal C_0$ denote its embedding in $\mathbb{CP}^2$. Then the following two conditions are equivalent:
    \begin{enumerate}
        \item[(1)] There exist distinct $A, B, C, D \in \mathcal C$ such that $\mathcal C$ is the isogonal cubic of $ABCD$.
        \item[(2)] The circular points at infinity (\cite{ref:Circle}) $I, J$ lie on $\mathcal C_0$, and the tangents at $I, J$ meet on $\mathcal C_0$.
    \end{enumerate}
\end{theorem}

In particular, the assertion that (2) is a \textit{sufficient} condition requires special care in $\mathbb R^2$ and $\mathbb{CP}^2$.

\subsection{Definitions and Conventions}

\begin{definition}[Isogonality in $\mathbb{RP}^2$]
    For points $P, A, B, C, D \in \mathbb R^2$, pairs of lines $(PA, PC), (PB, PD)$ are called \textit{isogonal} if they share the same pair of angle bisectors. If $P$ is a real point at infinity while $A, B, C, D$ remain in $\mathbb R^2$, we slightly modify our definition of isogonality to mean that for any line $\ell$ intersecting $PA, PB, PC, PD$ at points $E, F, G, H \in \mathbb R^2$, directed lengths $EF$ and $HG$ will be equal.
\end{definition}

The definition for points at infinity is equivalent to the midline of parallel lines $PA, PC$ being the same as the midline of parallel lines $PB, PD$, which complies with the idea of angles as a conceptual measure of ``distance" between two lines.

\begin{definition}[Quadrilateral Conventions and Isogonal Conjugates]
    We use the term ``quadrilateral" throughout this paper to refer to possibly self-intersecting quadrilaterals, whose vertices are distinct but possibly collinear. Points $P, Q$ are called \textit{isogonal conjugates} in a quadrilateral $ABCD$ if and only if $(AP, AQ), (AB, AD)$ are isogonal and the analogous holds for the other vertices.
\end{definition}

We will exclusively work in directed angles. For points $X, Y$, the notation $XY$ denotes the line $XY$ if $X$ and $Y$ are distinct, while $XY$ denotes the tangent at $X$ if $X \equiv Y$ and the context of the curve containing $X$ is clear (usually the isogonal cubic). The notation $(XYZ)$ denotes the circumcircle of $XYZ$ provided $X, Y, Z$ are distinct.

\begin{definition}[Notation for Intersection]
    We will let $\mathcal S \cap \mathcal T$ denote the intersection of sets of points $\mathcal S$ and $\mathcal T$, which is unique when $\mathcal S$ and $\mathcal T$ are distinct lines. When $\mathcal S$ is a cubic and $\mathcal T$ is a line $XY$ such that $X, Y \in \mathcal S$, we will use the notation $XY \cap \mathcal S$ to denote
    \begin{itemize}
        \item If either $X$ or $Y$ is a singular point, whichever one of $X, Y$ is singular
        \item If $X, Y$ are distinct and $XY$ is not tangent to $\mathcal S$, the third intersection of $XY$ with $\mathcal S$
        \item If $X, Y$ are distinct and $XY$ is tangent to $\mathcal S$, the tangency point of $XY$ with $\mathcal S$
        \item If $X, Y$ are not distinct and $X$ is not an inflection point of $\mathcal S$, the intersection of the tangent to $\mathcal S$ at $X$ with $\mathcal S$
        \item If $X, Y$ are not distinct and $X$ is an inflection point, the point $X$
    \end{itemize}
\end{definition}

These are essentially equivalent to the third intersection of $XY$ with $\mathcal C$ counting multiplicity.

We begin with this well-known characterization of all points with isogonal conjugates:

\begin{theorem}
    For fixed distinct points $A, B, C, D \in \mathbb R^2$ not all collinear, a point $P \in \mathbb{RP}^2$ is called \textit{excellent} if $(PA, PC)$ and $(PB, PD)$ are isogonal. Then $P$ is excellent if and only if it has an isogonal conjugate in $ABCD$.
\end{theorem}

Most proofs for this fact do not address the case when three of $A, B, C, D$ are collinear, so we will provide the full proof of this lemma for the sake of rigor.

\begin{proof}
    The first case is when, without loss of generality, $B, C, D$ are collinear. In this case, we need to prove the following: For triangle $ABC$ and $D \in BC$ and point $P$, the isogonal conjugate $Q$ of $P$ satisfies that $BC$ is a bisector of angle $\angle PDQ$ if and only if $(PA, PD)$, $(PB, PC)$ are isogonal.
    
    This is, in turn, equivalent to the following: For isogonal conjugates $P, Q$ in $ABC$, if $Q_A'$ be the reflection of $Q$ over $BC$, then $(PA, PQ_A')$, $(PB, PC)$ are collinear. To prove this, let $P, Q$ have pedal triangles $P_AP_BP_C$, $Q_AQ_BQ_C$ respectively; by \cite{ref:Three}, these share the same circumcircle $\omega$ centered at the midpoint of $PQ$. Let $PP_A$ meet $\omega$ at $R_A \ne P_A$; $PR_AQ_AQ_A'$ is a parallelogram, so
    \begin{align*}
        \angle Q_A'PC &= \angle Q_A'PP_A + \angle P_APC \\
        &= \angle Q_AR_AP_A + 90^\circ - \angle PCB \\
        &= \angle Q_AP_CP_A + 90^\circ - \angle PCB \\
        &= \angle Q_AP_CP_B + \angle BP_CP_A + 90^\circ - \angle PCB \\
        &= \angle Q_AQ_BQ_C + \angle BPP_A + 90^\circ - \angle PCB \\
        &= \angle Q_AQ_BQ + \angle QQ_BQ_C + 90^\circ - \angle CBP + 90^\circ - \angle PCB \\
        &= \angle BCQ + \angle QAB - \angle CBP - \angle PCB \\
        &= \angle PCA + \angle CAP + \angle BPC \\
        &= \angle BPC + \angle CPA \\
        &= \angle BPA
    \end{align*}
    as desired.
    
    \begin{figure}[!htbp]\centering
        \includegraphics[width=200pt]{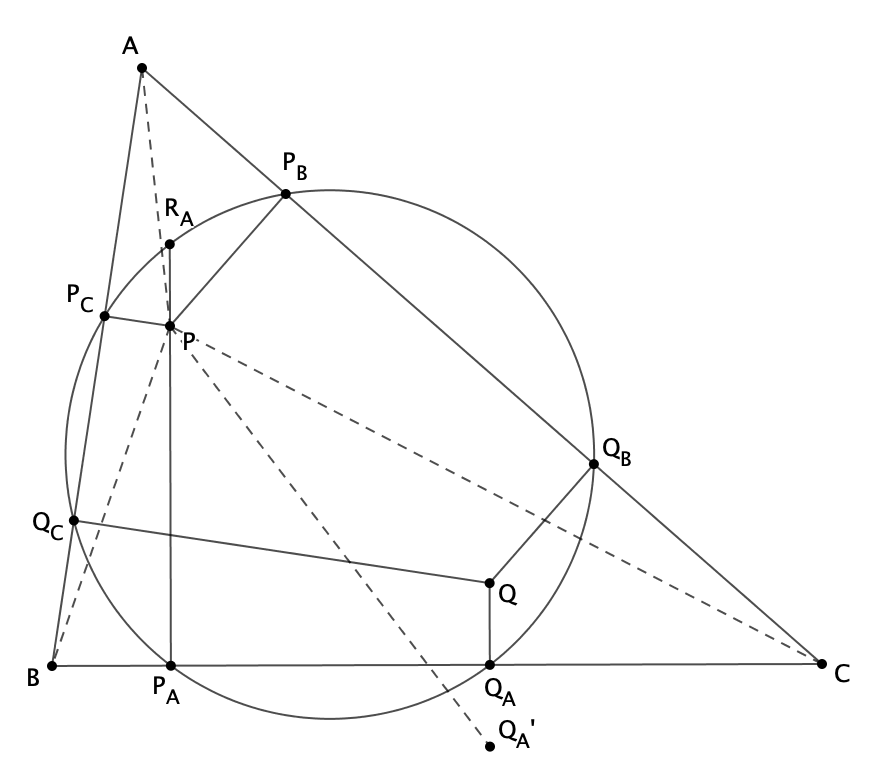}
        \caption{The Origin Lemma}
        \label{fig:0.1}
    \end{figure}
    
    Next, we prove this fact when no three of $A, B, C, D$ are collinear. While the proof for the general case is well-known, we will provide it for the sake of completion.
    \newline\newline
    \textbf{Lemma:} For quadrilateral $ABCD$ and point $P$, let the projections of $P$ onto $AB, BC, CD, DA$ be $E, F, G, H$. Prove that $EFGH$ is cyclic if and only if $\angle APB = \angle DPC$.

    \begin{subproof}
        With the various cyclic quadrilaterals,
        \begin{align*}
            \angle EFG + \angle GHE &= \angle EFP + \angle PFG + \angle GHP + \angle PHE \\
            &= \angle EBP + \angle PCG + \angle GDP + \angle PAE \\
            &= \angle APB + \angle CPD,
        \end{align*}
        which directly implies the desired statement.
    \end{subproof}

    Back to the main problem, drop perpendiculars $E, F, G, H$ from $P$ to $AB, BC, CD, DA$. First, we prove that if isogonal conjugate then $\angle APB = \angle DPC$. Let $P$ have isogonal conjugate $P'$ in $ABCD$. Then $P'$ is the isogonal conjugate of $P$ in both $YAB$ and $XAD$. Drop from $P'$ perpendiculars $E', F', G', H'$; by \cite{ref:Three} on $YAB$, $EFHE'F'H'$ is cyclic, and on $XAD$ we get $EGHE'G'H'$ is cyclic. In other words, $F, F', G, G'$ all lie on $(EE'HH')$, implying that $EFGH$ is cyclic, hence $\angle APB = \angle DPC$ as desired.
    
    Now, we prove that if $\angle APB = \angle DPC$ then it has an isogonal conjugate $P'$. Then $EFGH$ is cyclic; let its circumcircle meet $AB, BC, CD, DA$ at $E', F', G', H'$. By \cite{ref:Three}, the perpendiculars to $AB, BC, CD$ at $E', F', G'$ concur at a single point $P'$, the isogonal conjugate of $P$ in $XBC$. Analogously, the perpendiculars to $AB, CD, DA$ at $E', G', H'$ concur at the isogonal conjugate of $P$ in $XAD$. In other words, $P'H' \perp DA$ and is the isogonal conjugate of $P$ in both $XAD$ and $XBC$, implying that $P'$ is the desired isogonal conjugate of $P$ in $ABCD$.
\end{proof}

\subsection{Degenerate Cases}

One case where the locus of isogonal conjugates becomes degenerate is when $A, B, C, D$ are collinear on a line $\ell$. In this case, the locus becomes the line $\ell$ along with the circle centered on $\ell$ whose inversion swaps $A$ with $C$ and $B$ with $D$, if this circle exists.

Another case is when $ABCD$ is a parallelogram, where we have the following characterization:

\begin{theorem}
    If $ABCD$ is a parallelogram, the locus of excellent points is the line of infinity, along with the circumhyperbola passing through the points of infinity on the two angle bisectors of $\angle ABC$.
\end{theorem}

\begin{proof}
    By our extension of isogonality to $\mathbb{RP}^2$, the line of infinity is part of this locus. Then for all points $P \in \mathbb R^2$, by the Dual of Desargues' involution theorem (\cite{ref:Elementary}, 133), $(PA, PC)$, $(PB, PD)$ are isogonal if and only if angle $APC$ has the same angle bisectors as the pair of lines through $P$ parallel to $AB$ and $AD$. Thus if $P_1$ and $P_2$ are the points of infinity along with these angle bisectors of $\angle BAD$, we essentially need to find the locus $P \in \mathbb R^2$ for which the angle bisectors of $APC$ are parallel to $\ell_1$ and $\ell_2$.
    
    We claim that this locus is the hyperbola $\mathcal H$ centered at the midpoint $M$ of $AC$ passing through $P_1, P_2, A, C$. For any point $P \in \mathcal H$, $\mathcal H$ becomes the circumrectangular hyperbola of triangle $PAC$ centered at $M$, which is the isogonal conjugate of the perpendicular bisector of $AC$ wrt $PAC$. The isogonal conjugates of $P_1, P_2$ in $PAC$ then become the two arc midpoints of $AC$ in $(PAC)$, so $P_1, P_2$ are indeed the points of infinity along the angle bisectors of $PAC$.
    
    For the other direction, take any point $P$ such that $\angle APC$ has angle bisectors passing through $P_1, P_2$. Then the isogonal conjugate $\mathcal H'$ of the perpendicular bisector of $AC$ wrt $PAC$ will also pass through $P_1, P_2$, implying that $\mathcal H \equiv \mathcal H'$, so $P \in \mathcal H$, as desired. It is now clear that $B, D \in \mathcal H$, which completes the proof.
\end{proof}

For the rest of the paper, we will assume quadrilateral $ABCD$ does not fall under either of these cases. In particular, $A, B, C, D$ are not all collinear, and the midpoints of $AC, BD$ are distinct.

\section{Preliminary Lemmas}

Up until Section 6, we will work in $\mathbb{RP}^2$. All angles are directed mod $180^\circ$.

The following provides another well-known characterization of isogonal conjugates.

\begin{definition}
    Let $P$ be the spiral center of $ABCD$. For any point $Y$, call the unique point $Y'$ for which $P$ is the spiral center of $AYCY'$ the \textit{Spiral Inverse} of $Y$.
\end{definition}

\begin{theorem}
    The spiral inverse $X'$ of a excellent point $X$ is also the isogonal conjugate of $X$.
\end{theorem}

\begin{proof}
    Let $E = AD \cap BC, F = AB \cap CD$. If no three of $A, B, C, D$ are collinear, we have no problems, and otherwise we will assume without loss of generality that $A, B, D$ are collinear. Either way, the following relation is true:
    \[\angle DX'C = \angle DX'P + \angle PX'C = \angle XBP + \angle PAX = \angle APB + \angle BXA = \angle DEC + \angle CXD.\]
    Note that if $A, B, D$ are collinear, then we would have $B \equiv E$ and $D \equiv F$, but the above angle chase would still hold. Similarly, $\angle EX'C = \angle EDC + \angle CXE$, implying that $X, X'$ are isogonal conjugates in $CDE$.
    
    If no three of $A, B, C, D$ are collinear, analogously, $X, X'$ are isogonal conjugates in $BCF$, implying $X, X'$ are isogonal conjugates in $ABCD$, so we are done. Otherwise, under our WLOG that $A, B, D$ are collinear, $X, X'$ will be isogonal conjugates in $BCD$, and since $X$ is excellent, this implies $X, X'$ are isogonal conjugates in $ABCD$, as desired.
\end{proof}

\begin{definition}
    Denote by $\mathcal C$ the cubic which is the locus of all excellent points $X$.
\end{definition}

We will sometimes call $\mathcal C$ the ``isogonal cubic" throughout this paper.

\begin{proof}
    Proving that the locus is a cubic amounts to examining the equation
    \[\frac{\frac{d-x}{a-x}}{\frac{\overline{d-x}}{\overline{a-x}}} = \frac{\frac{c-x}{b-x}}{\frac{\overline{c-x}}{\overline{b-x}}}\]
    in the complex plane (\cite{ref:EGMO}, 6.1). Expanding this gives the desired third-degree equation in $x$. Note that the coefficients of 3rd degree coefficients $x^2\ol x, x\ol x^2$ in the expansion are both zero if and only if $a+c = b+d$, which confirms that parallelograms produce degenerate loci.
\end{proof}

For the rest of this paper, we will assume that $\mathcal C$ is non-degenerate.

Now, we may also recall the following well-known fact.

\begin{theorem}[Isogonal Conjugate at Infinity]\label{infinity}
    Let $M, N$ be the midpoints of $AC, BD$. Then the isogonal conjugate of $P$ is the point of infinity along $MN$.
\end{theorem}

\begin{proof}
    The parabola $\mathcal P$ tangent to the sides of $ABCD$ has focus $P$, and its directrix is the Gauss-Bodenmiller line (\cite{ref:Inconic}), which is perpendicular to $MN$ (\cite{ref:Gauss}).
    
    It is well known (\cite{ref:Three}) that for any conic with foci $X_1, X_2$ and any point $X$ for which tangents from $X$ exist, $XX_1$ and $XX_2$ are isogonal in the angle formed by the tangents from $X$ to the conic. Applying this to $X \equiv A$ and conic $\mathcal P$, we conclude that $AP$ and the perpendicular from $A$ to the directrix are isogonal in $\angle BAD$. Similar relations with $B, C, D$ imply the desired result.
\end{proof}

\begin{figure}[!htbp]\centering
    \includegraphics[width=350pt]{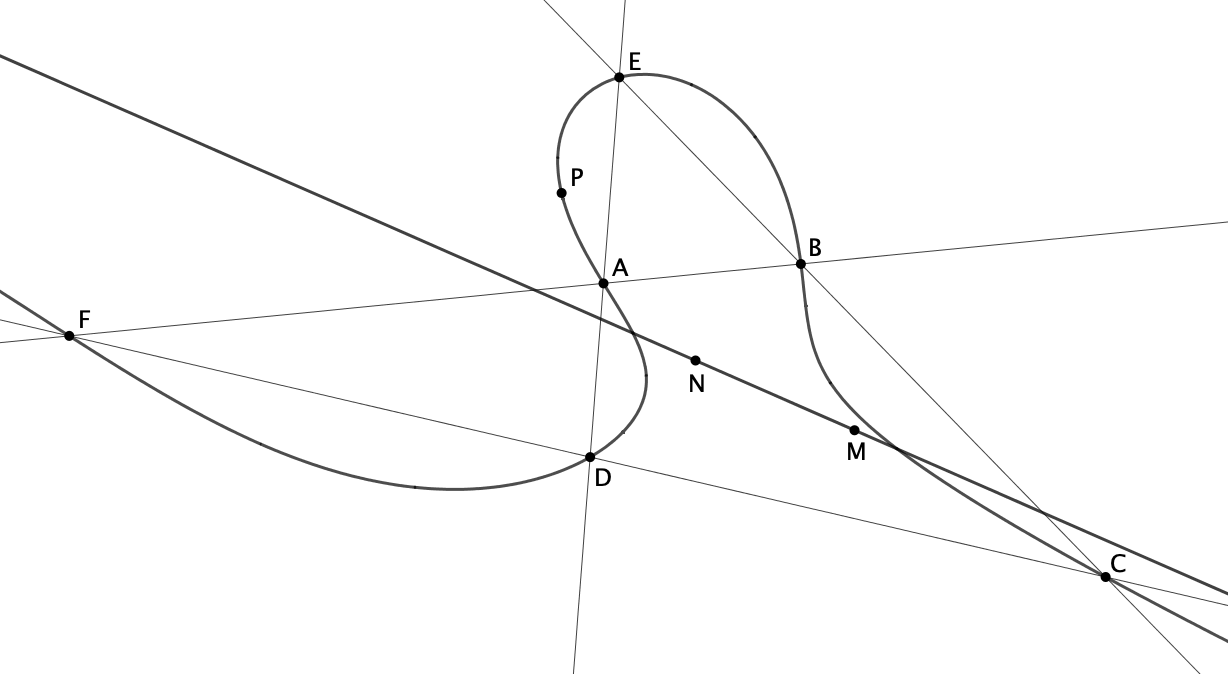}
    \caption{Main Configuration}
    \label{fig:1.1}
\end{figure}

\begin{theorem}\label{reversal}
    Consider two pairs $(X, X'), (Y, Y')$ of isogonal conjugates. Then $A, C$ are isogonal conjugates in $XYX'Y'$.
\end{theorem}

\begin{proof}
    Since $(AX, AX'), (AY, AY')$ are isogonal, $A, B, C, D$ are excellent in $XYX'Y'$. Since $A, C$ are spiral inverses in $XYX'Y'$, they are isogonal conjugates, as desired.
\end{proof}

The following corollary immediately follows.

\begin{corollary}\label{swallow-reversal}
    For isogonal conjugates $X, X'$ and excellent point $Y$, $(YX, YX'), (YA, YC)$ are isogonal.
\end{corollary}

This directly implies the following critical characterization:

\begin{corollary}[Generalization of Isogonal Cubic]\label{general}
    Consider two pairs $(X, X'), (Y, Y')$ of isogonal conjugates. Then the isogonal cubic of $ABCD$ is the isogonal cubic of $XYX'Y'$. Furthermore, any pair of isogonal conjugates $(K, L)$ in $ABCD$ are also isogonal conjugates in $XYX'Y'$.
\end{corollary}

\begin{proof}
    By \Cref{swallow-reversal}, for any point $Z$, if pairs of lines $(ZA, ZC), (ZB, ZD)$ are isogonal, then lines $(ZX, ZX'), (ZY, ZY')$ are also isogonal, so $ABCD$ and $XYX'Y'$ indeed share the same isogonal cubic. The second part then directly follows from \Cref{swallow-reversal}.
\end{proof}

\begin{figure}[!htbp]\centering
    \includegraphics[width=350pt]{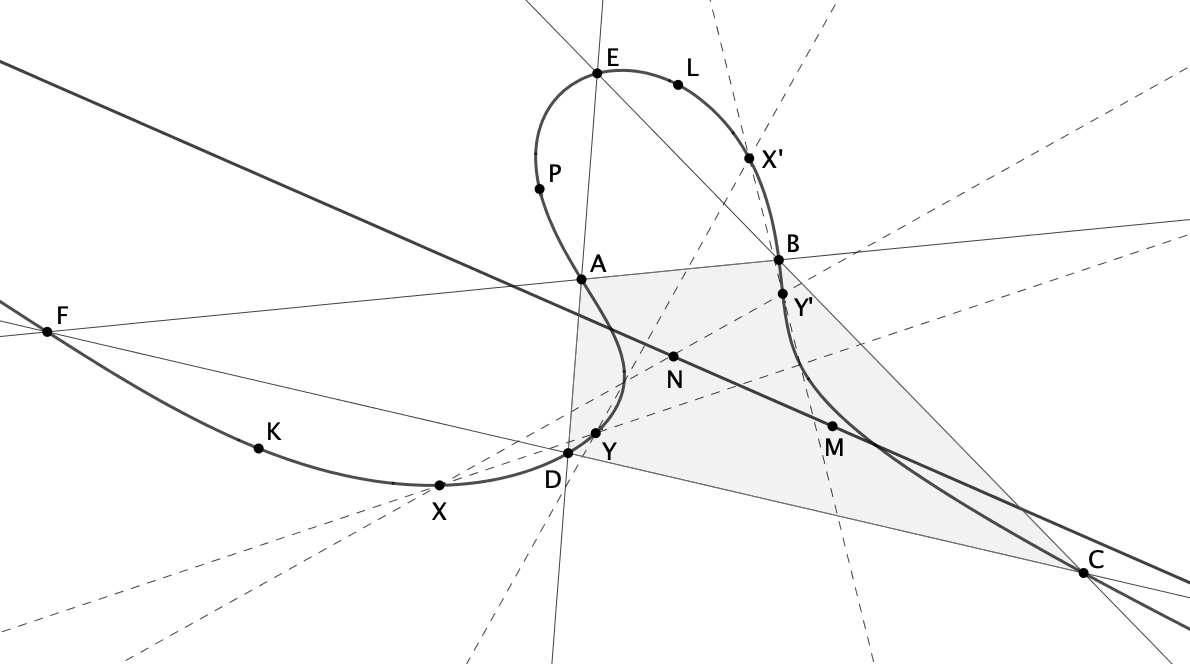}
    \caption{Quadrilateral Completeness}
    \label{fig:1.2}
\end{figure}

Thus the following is true by the Dual of Desargues' Involution Theorem on $XYX'Y'$:

\begin{corollary}[Quadrilateral Completeness]\label{completeness}
    For two pairs $(X, X'), (Y, Y')$ of isogonal conjugates, $XY \cap X'Y'$ and $XY' \cap X'Y$ lie on $\mathcal C$.
\end{corollary}

We now illustrate the relationship between isogonality and inconics.

\begin{theorem}
    Let $ABCD$ have inconic $\omega$ and isogonal conjugates $X, X'$. Then the tangents to $\omega$ from $X$ and $X'$ intersect at two pairs of isogonal conjugates.
\end{theorem}

\begin{proof}
    Call $IJKL$ the quadrilateral formed by these two pairs of tangents such that $IJ, JK, KL, LI$ are tangent to $\omega$. Since $\omega$ is an inconic of $XIX'K$ and the tangents from $A$ to $\omega$ ($AB$ and $AD$) are isogonal in $\angle XAX'$, by the Dual of Desargues' Involution on $XIX'K$ from $A$, $(AI, AK)$ are also isogonal in $\angle BAD$. Similar arguments imply $(I, K), (J, L)$ are isogonal conjugates as desired.
\end{proof}

\begin{corollary}\label{tangency-exists}
    For excellent point $X$ and isogonal conjugates $Y, Y'$, the line $XY'$ is tangent to the inconic of lines $AB, BC, CD, DA, XY$.
\end{corollary}

\section{Relationship of Inconics with Excellent Points}

Consider any excellent point $X$. Let $\omega$ be the inconic of $AB, BC, CD, DA, PX$; by \Cref{tangency-exists}, the line $\ell_1$ through $X$ parallel to $MN$ is also tangent to $\omega$. Reflect $\ell_1$ over the center of $\omega$ to get the second tangent $\ell_2$ from the point of infinity $\infty_{MN}$ along $MN$ to $\omega$; let $\ell_2$ meet the second tangent from $P$ to $\omega$ (other than $PX)$ at $X'$. By \Cref{tangency-exists}, $X$ and $X'$ are isogonal conjugates. Let $PX$ meet $\ell_2$ at $Y$, and let $PX'$ meet $\ell_1$ at $Y'$; then $Y, Y'$ are isogonal conjugates as well. We immediately get the following two corollaries:

\begin{figure}[!htbp]\centering
    \includegraphics[width=400pt]{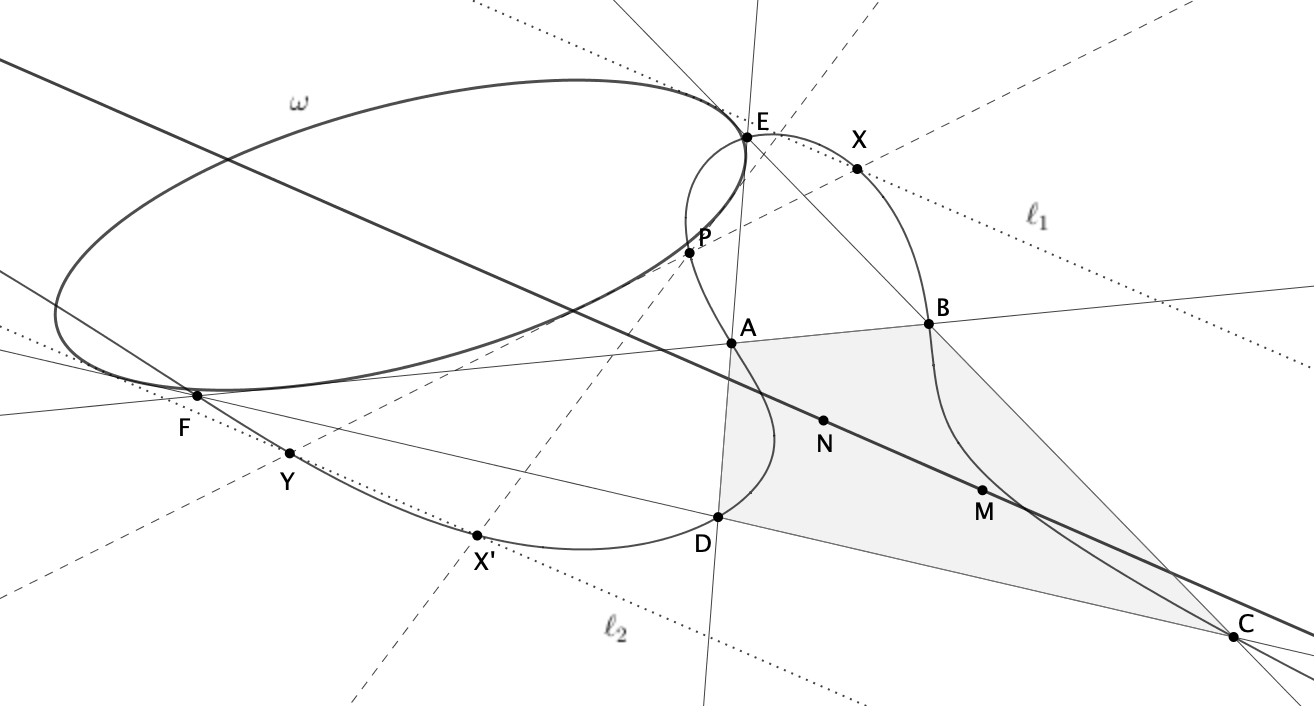}
    \caption{Tangents to an Inconic}
    \label{fig:2.1}
\end{figure}

\begin{theorem}\label{isogonal-MN}
    The midpoint of any two isogonal conjugates lies on $MN$.
\end{theorem}

\begin{theorem}\label{isogonal-parallel}
    For $X \in \mathcal C$ and $P_\infty$ the point of infinity along $MN$, let $Y = PX \cap \mathcal C$ and $X' = P_\infty Y \cap \mathcal C$. Then $X, X'$ are isogonal conjugates.
\end{theorem}

We may also note that the midpoint of $XY$ lies on $MN$. Since $\ell_1 \parallel \ell_2$, the bisectors of $\angle PX\infty_{MN}, \angle PY\infty_{MN}$ are parallel (perpendicular) to each other, with the perpendicular pairs of bisectors intersecting on $MN$. Thus, the following is a direct result.

\begin{theorem}[Parallel Bisectors]\label{rho-aias}
    If $X, Y$ lie on $\mathcal C$ with $XY$ passing through $P$, then the midpoint of $XY$ lies on $MN$, and the bisectors of $\angle AXC$ and $\angle AYC$ are parallel to each other.
\end{theorem}

\begin{figure}[!htbp]\centering
    \includegraphics[width=300pt]{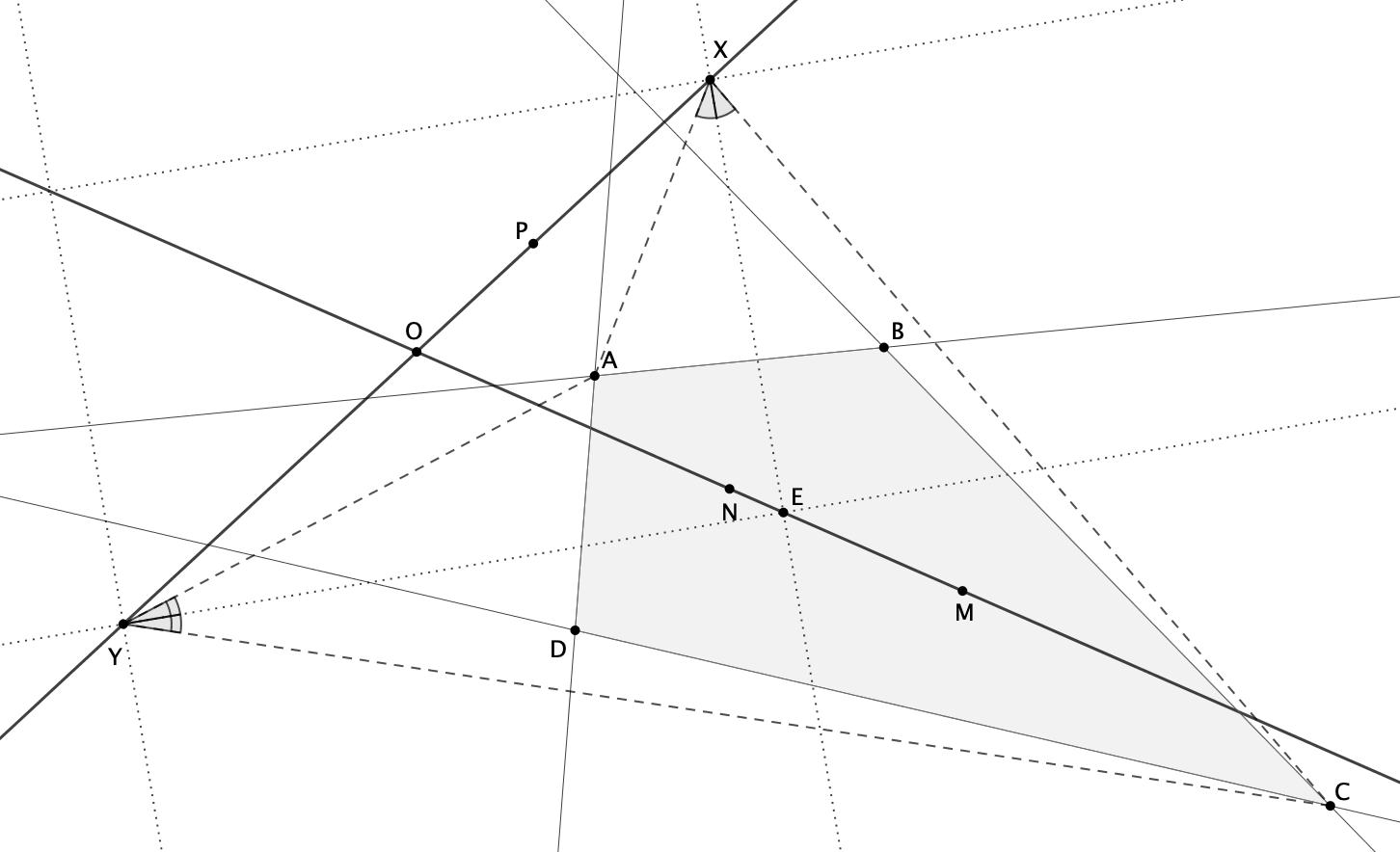}
    \caption{Angles with Parallel Bisectors}
    \label{fig:2.2}
\end{figure}

This produces a neat construction: if we are given an excellent point $X$, we may construct an excellent point $Y$ such that the bisectors of $\angle AXC$ and $\angle AYC$ are parallel to each other. By \Cref{rho-aias}, this is done by letting lines $PX$ and $MN$ intersect at a point $O$ and setting $Y$ to be the reflection of $X$ over $O$.

\section{Singular Points and Constructing Elements of the Isogonal Cubic}

We begin by noting that since $\mathcal C$ has real coefficients, for any two non-singular points $X$ and $Y$ on $\mathcal C$ in $\mathbb{RP}^2$, line $XY$ is either tangent to $\mathcal C$ at $X$ or $Y$, or $XY$ intersects $\mathcal C$ at a third point in $\mathbb{RP}^2$.

We begin with a well-known general lemma about cubics.

\begin{lemma}\label{four}
    From any point $X$ on general non-degenerate cubic $\mathcal C$, there are at most $4$ points $Y \in \mathcal C$ other than $X$ for which $XY$ intersects $\mathcal C$ at $Y$ with multiplicity 2.
\end{lemma}

\begin{proof}
    Note that if $X$ is singular, there are no such points $Y$, or else line $XY$ would intersect $\mathcal C$ at both $X$ and $Y$ with multiplicity 2, yielding $2+2 = 4$ total intersections. Assuming that $X$ is non-singular now, consider the embedding of $\mathcal C$ in $\mathbb{CP}^2$ with equation $F(x, y, z) = 0$. For any point $Y = (p: q: r)$ on $\mathcal C$, $Y$ is either a singular point, or the equation of the tangent at $Y$ is given by
    \[\frac{\partial F}{\partial x}(p, q, r) \cdot x + \frac{\partial F}{\partial y}(p, q, r) \cdot y + \frac{\partial F}{\partial z}(p, q, r) \cdot z = 0\]
    We want to $X = (x_0: y_0: z_0)$ to satisfy the above equation, so fixing $X$ gives us an equation in $p, q, r$ with degree $3-1 = 2$. Let $g(x, y, z)$ denote the expression
    \[\frac{\partial F}{\partial x}(x, y, z) \cdot x_0 + \frac{\partial F}{\partial y}(x, y, z) \cdot y_0 + \frac{\partial F}{\partial z}(x, y, z) \cdot z_0\]
    Regardless of whether $Y$ is a singular point of $\mathcal C$ or $XY$ is tangent to $\mathcal C$ at $Y$, all such points $Y$ will be solutions to the cubic $F(x, y, z) = 0$ and the conic $g(x, y, z) = 0$, which by Bezout's Theorem (\cite{ref:Fulton}, Section 5.3) gives at most $3(3-1)$ total solutions.
    
    Note that $X$ itself also satisfies both equations; we now claim that $X$ is actually a solution with multiplicity at least 2. Let $\ell$ be the tangent to $\mathcal C$ at $X$; then $\ell$ has equation
    \[\frac{\partial F}{\partial x}(x_0, y_0, z_0) \cdot x + \frac{\partial F}{\partial y}(x_0, y_0, z_0) \cdot y + \frac{\partial F}{\partial z}(x_0, y_0, z_0) \cdot z = 0\]
    To prove that $X$ is a solution to both $F$ and $g$ with multiplicity 2, we use the fact that $I_X(F, g) \ge m_X(F)m_X(g)$, where $I_X$ denotes the multiplicity of the intersection of curves $F$ and $g$ at $X$, and $m_X(F), m_X(g)$ denoting the multiplicity of point $P$ on curves $F, g$ (\cite{ref:Fulton}, Section 3.3).
    
    If $X$ is a singular point of $F$, then $I_X \ge 2$ as desired. Otherwise, equality holds iff the tangents at $X$ to $F$ and $g$ are distinct. Thus it suffices to show that $\ell$ is tangent to the conic $\mathcal H$ formed by $g$. To prove this, the tangent to $\mathcal H$ at $X$ is given by equation $Ax + By + Cz = 0$, where
    \[A = \frac{\partial^2 F}{\partial x^2}(x_0, y_0, z_0) \cdot x_0 + \frac{\partial^2 F}{\partial y^2}(x_0, y_0, z_0) \cdot y_0 + \frac{\partial^2 F}{\partial z^2}(x_0, y_0, z_0) \cdot z_0\]
    and $B, C$ are defined similarly. By Euler's Homogeneous Function Theorem (\cite{ref:Euler}),
    \[2 \cdot \frac{\partial F}{\partial x}(x_0, y_0, z_0) = \frac{\partial^2 F}{\partial x^2}(x_0, y_0, z_0) \cdot x_0 + \frac{\partial^2 F}{\partial y^2}(x_0, y_0, z_0) \cdot y_0 + \frac{\partial^2 F}{\partial z^2}(x_0, y_0, z_0) \cdot z_0\]
    which implies that
    \[A = 2 \cdot \frac{\partial F}{\partial x}(x_0, y_0, z_0), \quad\quad B = 2 \cdot \frac{\partial F}{\partial y}(x_0, y_0, z_0), \quad\quad C = 2 \cdot \frac{\partial F}{\partial z}(x_0, y_0, z_0)\]
    so the tangent to $\mathcal H$ at $X$ indeed has the same equation as $\ell$, as desired.
    
    Thus $X$ is a solution to $\mathcal C$ and $\mathcal H$ with multiplicity at least 2, so there are at most $3(3-1) - 2 = 4$ such points $Y$, as desired.
\end{proof}

The following lemma also better characterizes $\mathcal C$.

\begin{lemma}\label{truth}
    In $\mathbb{RP}^2$, $\mathcal C$ contains exactly one point at infinity.
\end{lemma}

\begin{proof}
    The embedding of $\mathcal C$ in $\mathbb{CP}^2$ will contain the circular points at infinity $I, J$ by virtue of being isogonal conjugates, so back in $\mathbb{RP}^2$ there can only be one real point at infinity. On the other hand, given $ABCD$ the point of infinity along the Newton-Gauss line will lie on $\mathcal C$, so there is exactly one.
\end{proof}

To better establish tangencies in $\mathcal C$, we first need to examine singular points.

\begin{theorem}[Singular Points on the Isogonal Cubic]\label{singular}
    A point $I \in \mathcal C$ is a singular point if and only if the isogonal conjugate of $I$ is itself.
\end{theorem}

\begin{proof}
    We remind our readers of our assumption in Section 1 that $\mathcal C$ is not degenerate.

    First, we prove that if $I$ is its own isogonal conjugate, then it is a singular point. Assume the contrary, that $I$ is not singular; then there are at most five non-singular points $X \in \mathcal C$ such that $XI$ is tangent to $\mathcal C$ at either $X$ or $I$. For all points $X$ such that $XI$ is \textit{not} tangent to $\mathcal C$, line $XI$ will intersect $\mathcal C$ at a point $Y \ne I, X$. By \Cref{swallow-reversal} this means that the line through $I, X, Y$ bisects angles $\angle AXC$, $\angle AYC$.
    
    In particular, this means line $CX$ is the reflection of line $AX$ over line $XY$, and line $CY$ is the reflection of line $AY$ over line $XY$, which implies that $C$ is the reflection of $A$ over $XY$. Thus $XI$ is the perpendicular bisector of $AC$. But line $XI$ rotates around $I$ as we vary $X$ along the cubic, contradicting the uniqueness of the perpendicular bisector of $AC$, the desired contradiction.
    
    Next, we prove that if $I$ is a singular point, then the isogonal conjugate of $I$ is itself. Assume the contrary; then let $J \ne I$ be the isogonal conjugate of $I$. Choose any isogonal conjugates $K, L$ distinct from $I, J$ (though we can set $K \equiv A$ etc). By \Cref{completeness}, $X = KI \cap LJ$ and $Y = KJ \cap LI$ will lie on $\mathcal C$. Since $I$ is a singular point, $KI$ will not intersect $\mathcal C$ at a point other than $K$ or $I$, so $X$ is either the same as $K$ or $I$.
    
    If $X \equiv I$, then $I, L, J$ are collinear. But since $I$ is singular line $ILJ$ intersects $\mathcal C$ at $I$ with multiplicity 2, so our assumption that $L \ne I, J$ implies that $I \equiv J$, the desired contradiction.
    
    Thus we must have $X \equiv K$, so $K, L, J$ are collinear. Thus we conclude $J$ lies on line $KL$ for any isogonal conjugates $K, L$. Choosing another pair $(R, S)$ of isogonal conjugates such that no three of $K, L, R, S$ are collinear, let $T = KR \cap LS$ and $U = KS \cap LR$; by \Cref{completeness}, $T$ and $U$ are isogonal conjugates in $\mathcal C$, so $KL, RS, TU$ concur at a single point $J$.
    
    Consider a conic $\mathcal H$ passing through $K, L, R, S$ but not tangent to line $TU$. Then the pole of line $TU$ in $\mathcal H$ is the intersection of $KL$ and $RS$, which is precisely $J$, which lies on line $TU$. For the pole of $TU$ in $\mathcal H$ to lie on $TU$ itself, $TU$ must be tangent to $\mathcal H$ at $J$, the desired contradiction.
\end{proof}

Now, we will construct the tangent to $\mathcal C$ at any non-singular point $X$ as follows.

\begin{theorem}[Tangent to Isogonal Cubic]\label{tangent}
    For isogonal conjugates $X, X'$, the isogonal $\ell$ to $XX'$ in $\angle AXC$ is tangent to $\mathcal C$ at $X$.
\end{theorem}

\begin{proof}
    As we move a point $Y \in \mathcal C$ with isogonal conjugate $Y'$, $XY$ and $XY'$ are isogonal in $\angle AXC$. Therefore, as $Y$ approaches $X'$, $Y'$ will approach $X$, eventually letting $XY'$ intersect $\mathcal C$ with multiplicity $2$ at $X$.
\end{proof}

In particular, the tangent to $\mathcal C$ at the point of infinity along $MN$ is given by the unique (by \Cref{truth}) asymptote of $\mathcal C$.

\begin{theorem}\label{tangents-isogonal}
    Let the tangents to $\mathcal C$ at isogonal conjugates $X, X'$ meet at $Y$, and let $XX'$ meet $\mathcal C$ at $Z \ne X, X'$. Then $Y, Z$ are isogonal conjugates.
\end{theorem}

\begin{proof}
    Let $Z^*$ be the isogonal conjugate of $Z$. By \Cref{reversal} $(XZ, XZ^*)$ are isogonal in $\angle AXC$, and $(X'Z, X'Z^*)$ are isogonal in $\angle AX'C$, so $Z^* \equiv Y$, as desired.
\end{proof}

\begin{lemma}
    The bisectors of $\angle AZC$ are perpendicular and parallel to $XX'$.
\end{lemma}

\begin{figure}[!htbp]\centering
    \includegraphics[width=350pt]{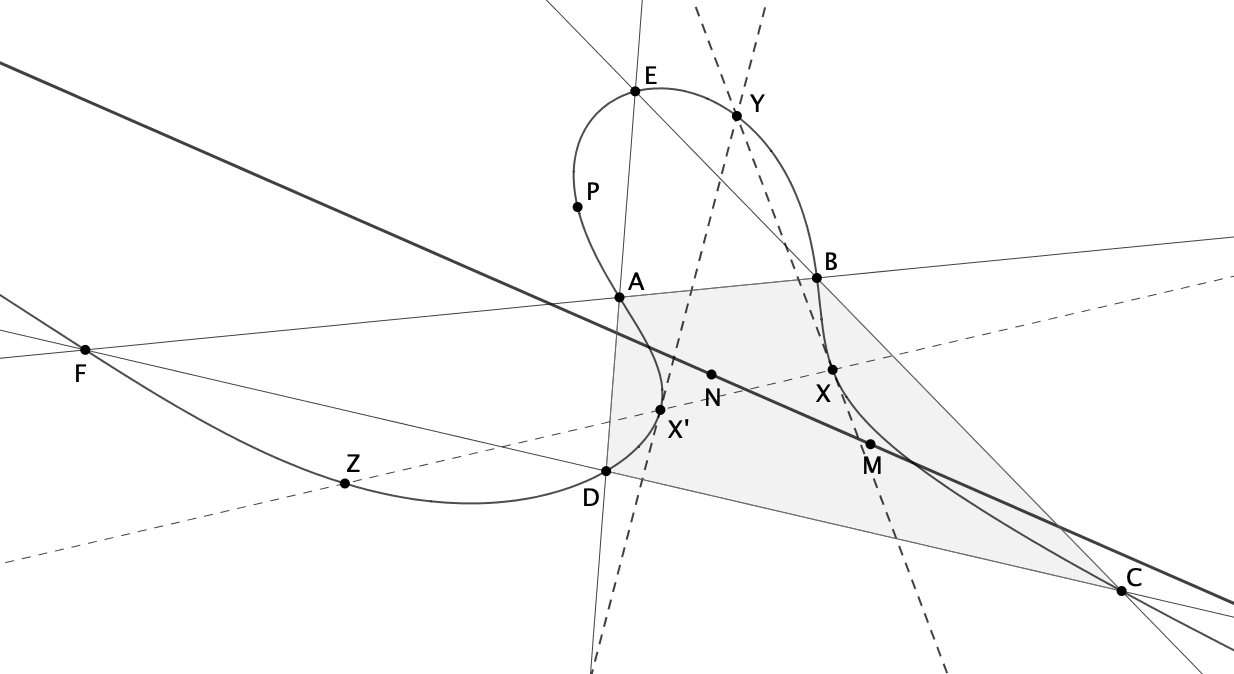}
    \caption{Tangents to the Cubic}
    \label{fig:3.1}
\end{figure}

\begin{proof}
    Examining isogonal conjugates $(A, C), (X, X')$, this follows from \Cref{swallow-reversal}.
\end{proof}

\begin{theorem}
    Let $PZ$ meet $\mathcal C$ at $W \ne Z$. Then $WX = WX'$.
\end{theorem}

\begin{proof}
    By \Cref{rho-aias}, the bisectors of $XWX'$ are perpendicular and parallel to $XX'$, which gives the desired result.
\end{proof}

\begin{corollary}
    If we denote $P$ by $0$ on the conic, then $W = X + X'$ under cubic addition (\cite{ref:Fulton}, Proposition 5.6.4). Thus, the cubic sum of any two isogonal conjugates $X, X'$ is equidistant from $X, X'$.
\end{corollary}

We thus obtain the following construction, if we desire to find all pairs of isogonal conjugates $(Y, Y')$ such that $YY'$ passes through a given excellent point $X$.

\begin{theorem}
    For $X \in \mathcal C$, let distinct $Y, Z \ne X$ lie on $\mathcal C$ such that $X, Y, Z$ are collinear and $XY$ bisects $\angle AXC$. Then $Y, Z$ are isogonal conjugates.
\end{theorem}

\begin{proof}
    Let $Y'$ be the isogonal conjugate of $Y$; then by \Cref{swallow-reversal}, $(XY, XY')$ and $(XA, XC)$ are isogonal. Since $XY$ bisects $\angle AXC$, $XY'$ and $XY$ must be the same line, implying that either $Y \equiv Y'$ or $Z \equiv Y'$.
    
    In the latter case we are done. In the former case, $Y$ must be an incenter or excenter of $ABCD$, so by \Cref{singular} $Y$ is a singular point. But $X, Y, Z$ are collinear and distinct despite line $XYZ$ intersecting $Y$ with multiplicity 2, the desired contradiction.
\end{proof}

\begin{figure}[!htbp]\centering
    \includegraphics[width=450pt]{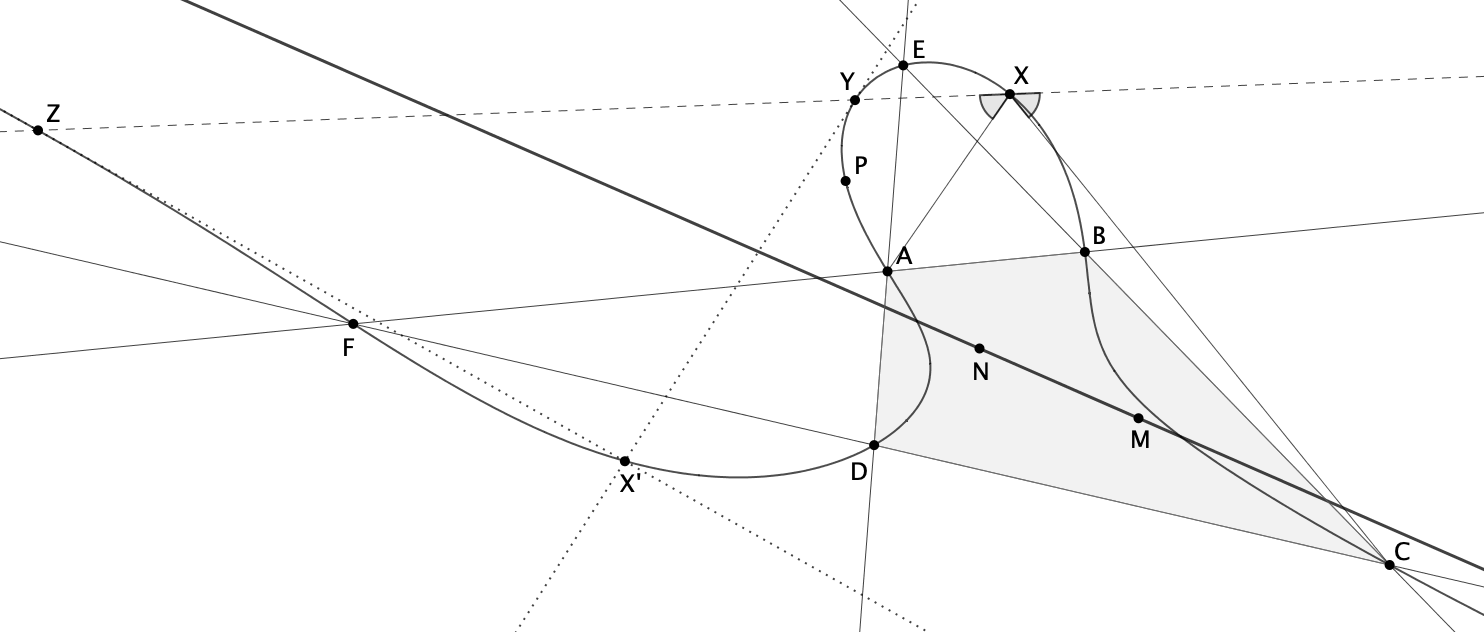}
    \caption{Isogonal Conjugates Collinear with a Given Point}
    \label{fig:3.2}
\end{figure}

Note that this also gives us a construction of points $Y$ on $\mathcal C$ such that $ZY$ is tangent to $\mathcal C$ at $Y$, where $Z$ is a fixed point on $\mathcal C$. This is done by letting $X$ be the isogonal conjugate of $Z$ and intersecting the angle bisectors of $\angle AXC$ with $\mathcal C$. By the above, there will be up to four such intersections $X_1, X_2, X_3, X_4$ on $\mathcal C$ for which $ZX_1, ZX_2, ZX_3, ZX_4$ are tangent to $\mathcal C$ at $X_1, X_2, X_3, X_4$.

\section{Constructing Intersections with Lines and Circles}

\begin{theorem}[Line Intersection]
    Consider excellent points $X, Y$. Denote by $Z$ the intersection of the reflections of $XY$ over the bisectors of $\angle AXC$ and $\angle AYC$. Then the intersection of $XY$ with $\mathcal C$ other than $X, Y$ is also the isogonal conjugate of $Z$. Furthermore, $PXYZ$ is cyclic.
\end{theorem}

\begin{proof}
    Let $W = XY \cap \mathcal C$, and let $W$ have isogonal conjugate $W'$. By \Cref{swallow-reversal}, $XW$ and $XW'$ are isogonal in $\angle AXC$, so line $XW'$ is the reflection of $XY$ over the bisectors of $\angle AXC$, implying that $W' \equiv Z$, proving that $XY \cap \mathcal C$ is indeed the isogonal conjugate of $Z$.
    
    To prove $PXYZ$ is cyclic, let $X', Y'$ be the isogonal conjugates of $ABC$. By \Cref{completeness}, $XY \cap \mathcal C$ lies on $X'Y'$, hence $W \in \mathcal X'Y'$. Then under spiral inversion, line $X'Y'W$ is mapped to the circumcircle of $XYZ$, which must pass through $P$, as desired.
\end{proof}

As a direct corollary, we have the following well-known theorem:

\begin{corollary}[Spiral Center of Isogonal Conjugates Lies on Circumcircle]
    For isogonal conjugates $(A, C)$, $(B, D)$ in $\triangle XYZ$, the spiral center of $ABCD$ lies on $(XYZ)$.
\end{corollary}

We may remark that this provides a construction of the intersection of $\mathcal C$ with any line $XY$, provided that $X$ and $Y$ lie on $\mathcal C$ themselves. Next, we characterize intersections of $\mathcal C$ with circles.

\begin{theorem}[Circle Intersection]
    Consider excellent points $E, F, G$ with isogonal conjugates $E', F', G'$. Then $(EFG)$ meets $\mathcal C$ at one other point which lies on $(EF'G'), (E'FG'), (E'F'G)$.
\end{theorem}

\begin{proof}
    There are two parts to this. First, we prove that if $H$ is a point on $\mathcal C$ such that $EFGH$ is cyclic, then $H$ lies on $(E'F'G)$ (which would imply that it lies on $(EF'G'), (E'FG')$ by symmetry). To prove this, by \Cref{reversal} with $EGE'G'$, since $H \in \mathcal C$, $\angle F'GE' = \angle EGF = \angle EHF = \angle F'HE'$.
    
    \begin{figure}[!htbp]\centering
        \includegraphics[width=400pt]{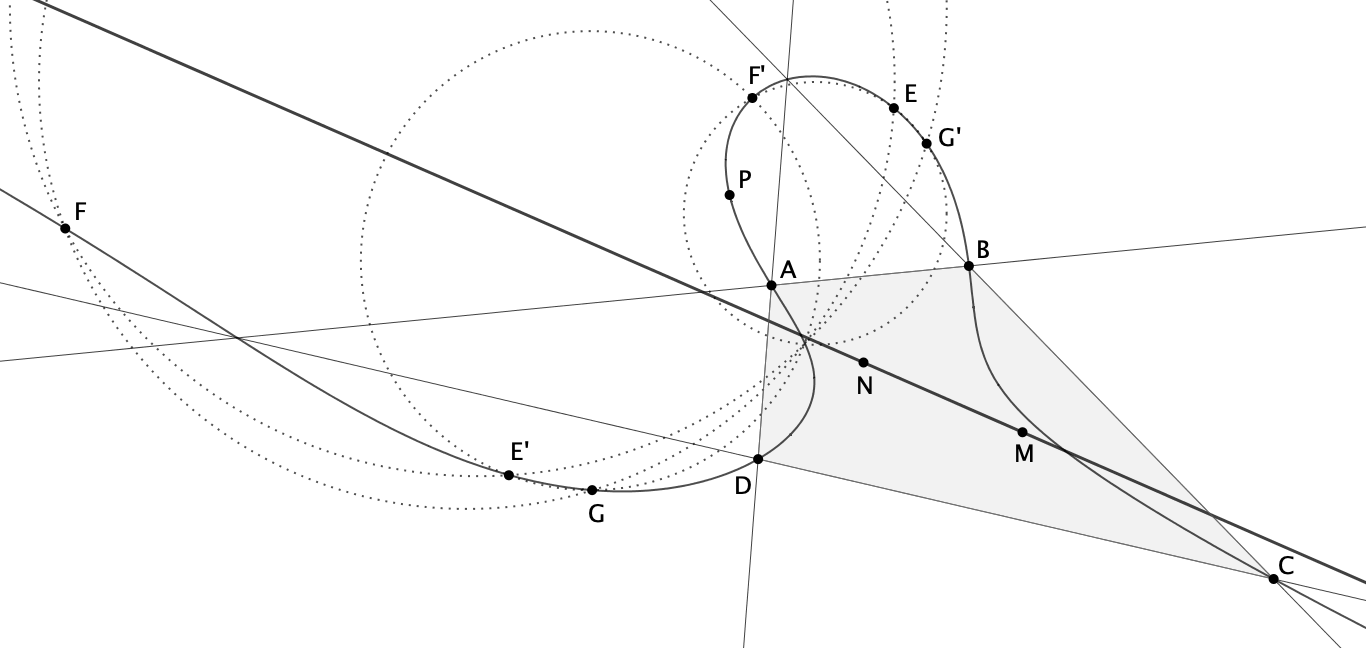}
        \caption{Intersecting with Circles}
        \label{fig:4.1}
    \end{figure}
    
    Next, we prove that if $H \equiv (EFG) \cap (E'F'G)$, then $\angle F'HE' = \angle F'GE' = \angle EGF = \angle EHF$, which implies that $E \in \mathcal C$ as desired.
\end{proof}

We may remark that this provides a construction for all points on $(EFG)$ lying on $\mathcal C$, provided $E, F, G$ lie on $\mathcal C$ themselves. The following is in fact true.

\begin{theorem}
    All circles intersect $\mathcal C$ in the real plane at at most 4 points.
\end{theorem}

\begin{proof}
    The circular points at infinity lie on $\mathcal C$ by virtue of being isogonal conjugates. The result follows from Bezout's Theorem, where curves of degree 2 and 3 meet for at most six points in $\mathbb{CP}^2$.
\end{proof}

\section{Characterizing the Isogonal Cubic}

For this section, we will work in $\mathbb{CP}^2$ and let $I, J$ denote the circular points at infinity. We must first extend the definition of isogonality to $\mathbb{CP}^2$ as follows:

\begin{definition}
    For distinct points $P, A, B, C, D \in \mathbb{CP}^2$, we call the two pairs of lines $(PA, PB)$ and $(PC, PD)$ \textit{isogonal} if and only if the three pairs of lines
    \[(PA, PB), \quad\quad (PC, PD), \quad\quad (PI, PJ)\]
    comprise a single involution, where $I, J$ are the circular points at infinity.
\end{definition}

One can check this complies with the angular definition of isogonality if $P, A, B, C, D \in \mathbb{RP}^2$.

\begin{corollary}
    For distinct points $A, B, C, D$ such that neither of $I, J$ lie on any of the lines $AB, BC, CD, DA$, the locus of points $X$ for which $(XA, XB), (XC, XD)$ are isogonal is a cubic (or curve of lesser degree) through $A, B, C, D, I, J$ in $\mathbb{CP}^2$.
\end{corollary}

\begin{proof}
    For any four points $A, B, C, D, E, F$, the locus of points $X$ for which $(XA, XB)$, $(XC, XD)$, $(XE, XF)$ comprise a single involution is a cubic through $A, B, C, D, E, F$. Setting $E, F$ as the circular points at infinity gives the desired result.
\end{proof}

Hence we will call a non-degenerate cubic $\mathcal C$ the "isogonal cubic" of quadrilateral $ABCD$ if it is the locus of all points $X$ for which $(XA, XC), (XB, XD)$ are isogonal (using the new definition).

\begin{corollary}[Loci of Isogonality]\label{loci}
    If the locus of points $X$ for which $(XA, XC), (XB, XD)$ are isogonal is a non-degenerate cubic, then neither $I$ nor $J$ cannot lie on any of the lines $AB, BC, CD, DA$.
\end{corollary}

\begin{proof}
    Assume the contrary, that WLOG $I \in AB$. Then for any point $P$ on line $AB$, pairs $(XA, XC)$, $(XB, XD)$, $(XI, XJ)$ are part of a single degenerate involution. Thus the locus of points $X$ for which $(XA, XC)$, $(XB, XD)$ are isogonal includes line $AB$, contradicting the proposition that the locus is a non-degenerate cubic.
\end{proof}

In other words, if $ABCD$ has a non-degenerate isogonal cubic $\mathcal C$, then $I$ and $J$ will not lie on $AB, BC, CD, DA$.

Now, the main result of this paper is the following:

\begin{theorem}[Characterization of all Isogonal Cubics]
    Let $\mathcal C$ be a non-degenerate cubic in $\mathbb{CP}^2$ containing circular points at infinity $I$ and $J$ at non-singular points. Then the following two conditions are equivalent:
    \begin{enumerate}
        \item[(1)] There exist non-singular $A, B, C, D \in \mathcal C$ such that $\mathcal C$ is the isogonal cubic of $ABCD$.
        \item[(2)] The tangents to $\mathcal C$ at $I, J$ intersect each other on $\mathcal C$.
    \end{enumerate}
\end{theorem}

We begin with the following direct result of Cayley-Bacharach (\cite{ref:Cayley}).

\begin{lemma}[Cubics Containing Complete Quadrilateral]\label{complete-quad}
    For $P, Q$ on non-degenerate cubic $\mathcal C$, consider $T \in \mathcal C$ and let $U = PT \cap \mathcal C$, $V = QT \cap \mathcal C$ such that $P, Q, T, U, V$ are non-singular. Then $PV \cap QU \in \mathcal C$ iff $PP \cap QQ \in \mathcal C$.
\end{lemma}

\begin{proof}
    Let $X = PP \cap QQ, Y = PV \cap QU$. Cayley-Bacharach on triples of lines $(XPP, QTV, QUY)$, $(XQQ, PTU, PVY)$ completes both directions.
\end{proof}

\begin{lemma}[Locus of Involution]\label{involution-locus}
    Consider distinct points $A, B, C, D, E, F$ in general position and $G = AC \cap BD, H = AD \cap BC, I = AE \cap BF, J = AF \cap BE$, such that none of the ten points are the circular points at infinity. Then there is a unique cubic $\mathcal C$ through these ten points. Furthermore, for every $P \in \mathcal C$, we have
    \[(PA, PB), \quad (PC, PD), \quad (PE, PF), \quad (PG, PH), \quad (PI, PJ)\]
    are part of a fixed involution.
\end{lemma}

\begin{figure}[!htbp]\centering
    \includegraphics[width=200pt]{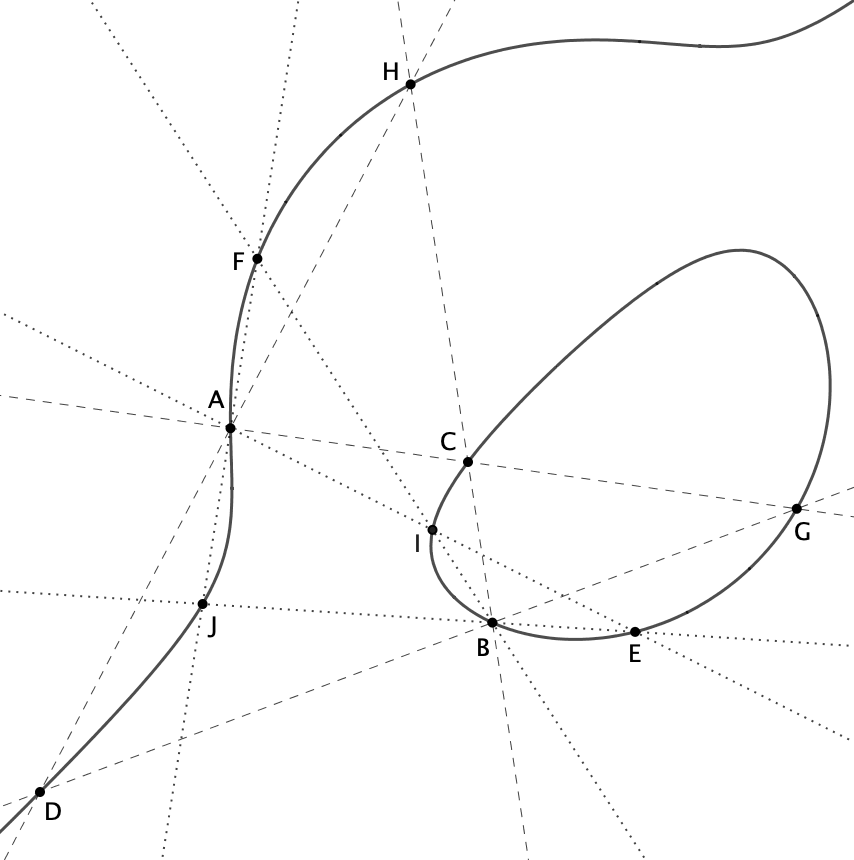}
    \caption{Two Complete Quadrilaterals}
    \label{fig:5.1}
\end{figure}

\begin{proof}
    By the Dual of Desargues' Involution Theorem, the locus $\mathcal C$ of all points $P$ for which $(PA, PB)$, $(PC, PD)$, $(PE, PF)$ are part of a single involution is a cubic through
    \[A, B, C, D, E, F, G, H, I, J.\]
    Thus, there exists a cubic through these $10$ points. Since $A, B, C, D, E, F$ are in general position, no four of the $10$ constructed points are collinear. Since $\mathcal C$ passes through these $10$ fixed points, the cubic through these $10$ points must be unique, as desired.
\end{proof}

We are finally set up to prove the main result.

\begin{theorem}[Characterization in $\mathbb{CP}^2$, Condition (2) $\implies$ (1)]\label{class-forward}
    Let $\mathcal C$ be a non-degenerate cubic through $I, J$ such that $I$ and $J$ are non-singular, and $II$ intersects $JJ$ at a point $X$ on $\mathcal C$. Then there exist non-singular points $A, B, C, D \in \mathcal C$ apart from $I, J$ such that $\mathcal C$ is the isogonal cubic of $ABCD$.
\end{theorem}

\begin{proof}
    Choose any point $A \in \mathcal C$. Let $B = IA \cap \mathcal C, D = JA \cap \mathcal C$; by \Cref{complete-quad}, $ID$ and $JB$ meet a point $C$ on $\mathcal C$. Construct four points $A', B', C', D' \in \mathcal C$ distinct from $A, B, C, D$ analogously, where $I = A'B' \cap C'D'$ and $J = A'D' \cap B'C'$. We may select $A, A'$ such that none of $A, A', B, B', C, C', D, D'$ are singular.
    
    By \Cref{involution-locus}, $\mathcal C$ is the locus of points $P$ for which $(PI, PJ), (PA, PC), (PA', PC')$ are part of a single involution. But since this involution concerns the circular points at infinity, it follows that $\mathcal C$ is the locus for which $(PA, PC), (PA', PC')$ are isogonal. We are done by taking quadrilateral $AA'CC'$.
\end{proof}

\begin{theorem}[Characterization in $\mathbb{CP}^2$, Condition (1) $\implies$ (2)]\label{class-reverse}
    Let $\mathcal C$ be the non-degenerate isogonal cubic of $ABCD$ where $A, B, C, D$ are non-singular. Suppose that $I, J$ lie on $\mathcal C$ at non-singular points distinct from $A, B, C, D$. Then $II \cap JJ \in \mathcal C$.
\end{theorem}

\begin{proof}
    Let $X = AI \cap CJ$ and $Y = AJ \cap CI$; then $X, Y$ are non-singular. Note that $(XA, XC)$ and $(XI, XJ)$ are the same pair of lines, so $(XA, XC)$, $(XB, XD)$, $(XI, XJ)$ form an involution. (If $X$ is the same point as either $A, C, I,$ or $J$, we instead use the tangent to $\mathcal C$ at $X$ when necessary.)
    
    So by definition, $X \in \mathcal C$; similarly, $Y \in \mathcal C$. By \Cref{complete-quad}, this implies $II \cap JJ \in \mathcal C$ as desired.
\end{proof}

Going back to $\mathbb R^2$, we derive the complete characterization of all non-degenerate isogonal cubics:

\begin{theorem}[Characterization of All Isogonal Cubics in $\mathbb R^2$]\label{class-real}
    Let $\mathcal C$ be a non-degenerate cubic in $\mathbb R^2$, and let $\mathcal C_0$ denote its embedding in $\mathbb{CP}^2$. Then the following two conditions are equivalent:
    \begin{enumerate}
        \item[(1)] There exist distinct $A, B, C, D \in \mathcal C$ such that $\mathcal C$ is the isogonal cubic of $ABCD$.
        \item[(2)] The circular points at infinity $I, J$ lie on $\mathcal C_0$, and the tangents to $\mathcal C_0$ at $I, J$ intersect each other on $\mathcal C_0$.
    \end{enumerate}
\end{theorem}

\begin{proof}
    The only aspects of the proof we need to modify for this new wording are to prove that:
    \begin{enumerate}
        \item Under the conditions of (1), if $\mathcal C$ is the isogonal cubic of $ABCD$ where $A, B, C, D$ are distinct, then $A, B, C, D$ cannot be singular points of $\mathcal C_0$.
        \item Under the conditions of (2), if any cubic $\mathcal C$ in $\mathbb R^2$ satisfies that its embedding $\mathcal C_0$ in $\mathbb{CP}^2$ passes through $I$ and $J$, then $I$ and $J$ are not singular.
        \item Under the conditions of (2), for $A \in \mathcal C$ such that $K = AI \cap \mathcal C_0, L = AJ \cap \mathcal C_0, A' = IL \cap JK$ all lie on distinct points of $\mathcal C_0$, then the point $A'$ will be contained in $\mathcal C$ as well.
    \end{enumerate}
    
    Let $\mathcal C$ have Cartesian equation $ax^3+bx^2y+cxy^2+dy^3 + G(x, y) = 0$, where $G$ is a second-degree polynomial in $x, y$. Then $a, b, c, d$ must be real, and since $\mathcal C$ is not degenerate, they cannot all be zero. Thus $\mathcal C_0$ has equation $F(x, y, z) = 0$ where $F(x, y, z) = ax^3 + bx^2y + cxy^2 + dy^3 + zP(x, y, z)$, where $P(x, y, z)$ is a second-degree homogeneous polynomial in $x, y, z$.
    
    \begin{subproof}
        For (a), assume that $\mathcal C$ is the isogonal cubic of $ABCD$. Note that $A$ is a singular point in $\mathcal C$ iff it is a singular point in $\mathcal C_0$, because both are equivalent to $\frac{\partial F}{\partial x}(A) = \frac{\partial F}{\partial y}(A) = \frac{\partial F}{\partial z}(A) = 0$, the same equation in both $\mathbb{RP}^2$ and $\mathbb{CP}^2$. We just need to show that $A$ is not a singular point in $\mathbb R^2$.
        
        By \Cref{singular}, $A$ is singular if and only if $A$ is the isogonal conjugate of itself in $ABCD$. But the isogonal conjugate of $A$ is $C$, and since $A$ and $C$ are distinct, this cannot happen. Therefore, $A$ and similarly $B, C, D$ are not singular points of $\mathcal C_0$, as desired. This proves part (a).
    \end{subproof}
    
    \begin{subproof}
        For (b), we consider general cubic $\mathcal C$ which contains $I, J$. Plugging in $I = (1: i: 0)$ yields equation $a+bi-c-di = 0$, which implies that $a = c$ and $b = d$ because $a, b, c, d$ are all real. Thus $\mathcal C$ has equation $(x^2+y^2)(ax+by) + zP(x, y, z)$. We get $\frac{\partial F}{\partial x}(1: i: 0) = 3ax^2 + 2xby + ay^2$.
        
        Assume, for the sake of contradiction, that $I$ is a singular point. We require $\frac{\partial F}{\partial x} = 0$ for $(1: i: 0)$, which rearranges to $2a + 2bi = 0$. Since $a, b$ are real, this implies that $a = b = 0$, so $a, b, c, d$ are all zero - the desired contradiction. This proves part (b).
    \end{subproof}
    
    \begin{subproof}
        For (c), it suffices to show that $A' \in \mathbb R^2$. Note that $I, J, A, K, L, A'$ all lie on $\mathcal C_0$, which has all real coefficients. Now, $K, L$ do not lie on the line of infinity, or else $A$ would lie on the line of infinity, which would imply $\mathcal C_0$ containing four points on a line and thus be degenerate, a contradiction. Thus $K, L$ they are contained in $\mathbb C^2$ and thus can be expressed in Cartesian coordinates $(k_x, k_y)$ and $(l_x, l_y)$ respectively. Since $A \in \mathbb R^2$, note that $K$ and $L$ cannot lie in $\mathbb R^2$ - or else the entire lines $AK$ and $AL$ will be contained in $\mathbb{RP}^2$ and never intersect the line of infinity at complex points $I$ and $J$.
        
        From part (b), $\mathcal C$ must have equation of the form $ax^3+bx^2y+axy^2+by^3 + G(x, y) = 0$ where $a, b$ and the coefficients of $G$ are real numbers. For $K$ to satisfy this equation, the point $K'$ whose Cartesian coordinates are the complex conjugates of $K$ - that is, $K' = \left(\ol{k_x}, \ol{k_y}\right)$ in Cartesian coordinates - must also satisfy this equation, and thus lie on $\mathcal C_0$. Having started with $A, I, K$ collinear, we now claim that $A, J, K'$ are collinear. It suffices to show that
        \[\begin{vmatrix} 1 & -i & 0 \\ a_x & a_y & 1 \\ \ol{k_x} & \ol{k_y} & 1 \end{vmatrix} = 0 \quad\quad \text{given that} \quad\quad \begin{vmatrix} 1 & i & 0 \\ a_x & a_y & 1 \\ k_x & k_y & 1 \end{vmatrix} = 0\]
        Letting $k_x = p+qi$ and $k_y = r+si$ for $p, q, r, s \in \mathbb R$, the second determinant equation gives us
        \[0 = \begin{vmatrix} 1 & i & 0 \\ a_x & a_y & 1 \\ p+qi & r+si & 1 \end{vmatrix} = -r-si-q+pi+a_y-ia_x\]
        where $(a_x, a_y)$ are the Cartesian coordinates of $A$. Equating the real and imaginary parts yields $a_y = q+r, a_x = p-s$. Similarly, the first determinant equation gives us
        \[0 = \begin{vmatrix} 1 & -i & 0 \\ a_x & a_y & 1 \\ p-qi & r-si & 1 \end{vmatrix} = -r+si-q-pi+a_y+ia_x\]
        and equating the real and imaginary parts yields $a_y = q+r, a_x = p-s$ - the exact same conditions. Therefore, given that $A, I, K$ are collinear, we indeed conclude that $A, J, K'$ are collinear.
        
        In other words, $K'$ is the unique intersection of $\mathcal C_0$ with $AJ$, hence $K' \equiv L$. Thus $A' = IK' \cap JK$, and $A'$ will not be a point at infinity (otherwise $K$ will also be a point at infinity). Hence in quadrilateral $AKA'K' \in \mathbb C^2$, we have $AK$ meets $A'K'$ at a point of infinity, and $AK'$ meets $A'K$ at a point of infinity - so complex segments $AA'$ and $KK'$ share the same midpoint $M$. Letting $A$ have Cartesian coordinates $(m, n)$ in $\mathbb C^2$, this means that $M$ has Cartesian coordinates
        \[\left(\frac{a_x+m}{2}, \frac{a_y+n}{2}\right) \quad = \quad \left(\frac{k_x + \ol{k_x}}{2}, \frac{k_y + \ol{k_y}}{2}\right)\]
        But $\frac{k_x+\ol{k_x}}{2}$ is just the real part of $k_x$, so the coordinates of $M$ are real as well. Hence $m$ and $n$ are real, so $A' = (m, n)$ lies in $\mathbb R^2$, proving part (c).
    \end{subproof}
    
    With (a), (b), (c) proven, for the sake of completion we will show how this fully finishes our characterization. For the direction (1) $\implies$ (2), we start with $ABCD$, and by (a) none of $A, B, C, D$ are singular points. Then the result directly follows from \Cref{class-reverse}.
    
    For the direction (2) $\implies$ (1), we start with circular points at infinity $I, J$ lying on $\mathcal C_0$, which by (b) implies that $I, J$ are not singular points. Assuming that the tangents to $\mathcal C_0$ at $I$ and $J$ intersect each other on $\mathcal C$, we can choose any point $A \in \mathcal C$ and letting $K = AI \cap \mathcal C_0$, $L = AJ \cap \mathcal C_0$, and $A' = IL \cap JK$ where $A' \in \mathcal C_0$, and then choose another point $B \in \mathcal C$ and define $B' \in \mathcal C_0$ the same way, such that all points formed by these intersections are distinct. By \Cref{class-forward}, $\mathcal C_0$ will be the isogonal cubic of $ABA'B'$. In addition, (c) implies that $A'$ and $B'$ will in fact lie in $\mathbb R^2$ as well. Therefore, $ABA'B'$ is fully contained in $\mathbb R^2$, so $\mathcal C$ is indeed the isogonal cubic of $ABA'B'$. This completes the solution.
\end{proof}

\section{Uniqueness in the Isogonal Cubic}

With this algebraic characterization of all isogonal cubics in $\mathbb R^2$ in mind, in this section, we prove that given an isogonal cubic $\mathcal C \in \mathbb{RP}^2$, there is only one possible spiral center $P$, and for any $X \in \mathcal C$, there is only one possible point that could be the isogonal conjugate of $X$.

\begin{theorem}[Uniqueness of the Spiral Center]
    Consider non-degenerate $\mathcal C \in \mathbb{RP}^2$ such that there exist $A, B, C, D \in \mathbb R^2$ for which $\mathcal C$ is the isogonal cubic of $ABCD$. Let $ABCD$ have spiral center $P$. Let $\mathcal C_0$ denote the embedding of $\mathcal C$ in $\mathbb{CP}^2$. Then $PI$ and $PJ$ are respectively tangent to $\mathcal C_0$ at $I$ and $J$.
\end{theorem}

\begin{proof}
    Assume, for the sake of contradiction, that $PI$ is not tangent to $\mathcal C_0$ at $I$; then $PJ$ cannot be tangent to $\mathcal C_0$ at $J$ either, so by part (c) of \Cref{class-real}, $PI$ and $PJ$ intersect $\mathcal C_0$ at $K, L \in \mathbb C^2$ respectively, distinct from $I, J, P$, and $IL$ and $JK$ intersect $\mathcal C_0$ at $Q \in \mathbb R^2$. Then by \Cref{class-forward}, $\mathcal C$ is the non-degenerate isogonal cubic of the three quadrilaterals $ABCD$, $APCQ$, $BPDQ$.
    
    By \Cref{truth}, there is one point of infinity $P_\infty \in \mathcal C$, which is the point of infinity along the Newton-Gauss lines of $ABCD$, $APCQ$, $BPDQ$. Let $AP$ meet $\mathcal C$ at $E$; then by \Cref{isogonal-parallel}, $C, E, P_\infty$ are collinear. Since $\mathcal C$ is the isogonal cubic of $APCQ$, it follows that $AP \cap CQ \in \mathcal C$, so in fact $Q = CE \cap \mathcal C$. Since $C, E$ lie in $\mathbb R^2$ they are distinct from $P_\infty$.
    
    If $C, E, P_\infty$ are distinct, then $Q \equiv P_\infty$, contradicting $Q \in \mathbb R^2$, as desired. So line $CE$ intersects $\mathcal C$ with multiplicity 2. We assumed that $ABCD \in \mathbb R^2$, so we cannot have $C \equiv P_\infty$, hence either $C \equiv E$ or $E \equiv P_\infty$. In either case, we cannot have $Q \equiv P_\infty$ else $Q \in \mathbb R^2$ is contradicted; thus $Q \equiv C$. But considering $A, P$ are the respective isogonal conjugates of $C, Q$ in $APCQ$, so this implies $A \equiv P$. Now, the isogonal conjugates of $A, P$ in $ABCD$ are $C, P_\infty$, which implies that $C \equiv P_\infty$ - the desired contradiction.
\end{proof}

In other words, a given non-generate isogonal cubic can only have one possible spiral center - we may now call this \textit{the} spiral center of a given isogonal cubic $\mathcal C$. This leads to the following result, allowing us to define isogonal conjugation on any given isogonal cubic without having to construct a base quadrilateral $ABCD$:

\begin{theorem}[Uniqueness of the Isogonal Conjugate]
    Consider non-degenerate $\mathcal C \in \mathbb {RP}^2$ such that there exist $A, B, C, D \in \mathbb R^2$ for which $\mathcal C$ is the isogonal cubic of $ABCD$. Then for any point $X \in \mathcal C$, there is only one possible point $X' \in \mathcal C$ which could be the isogonal conjugate of $X$ in $ABCD$.
\end{theorem}

\begin{proof}
    Let $P$ be the spiral center of $\mathcal C$, and let $P_\infty$ be the point of infinity of $\mathcal C$. Consider any $X \in \mathcal C$. If $X \equiv P$, its isogonal conjugate is $P_\infty$, and vice versa. If $X$ is neither $P$ nor $P_\infty$, let $Y = PX \cap \mathcal C$ and $X' = P_\infty Y \cap \mathcal C$. Then by \Cref{isogonal-parallel}, $X'$ is the isogonal conjugate of $X$ in $ABCD$ no matter which $ABCD$ we choose. Since $P$ is fixed, $X'$ depends only on $X$, as desired.
\end{proof}

\begin{figure}[!htbp]\centering
    \includegraphics[width=300pt]{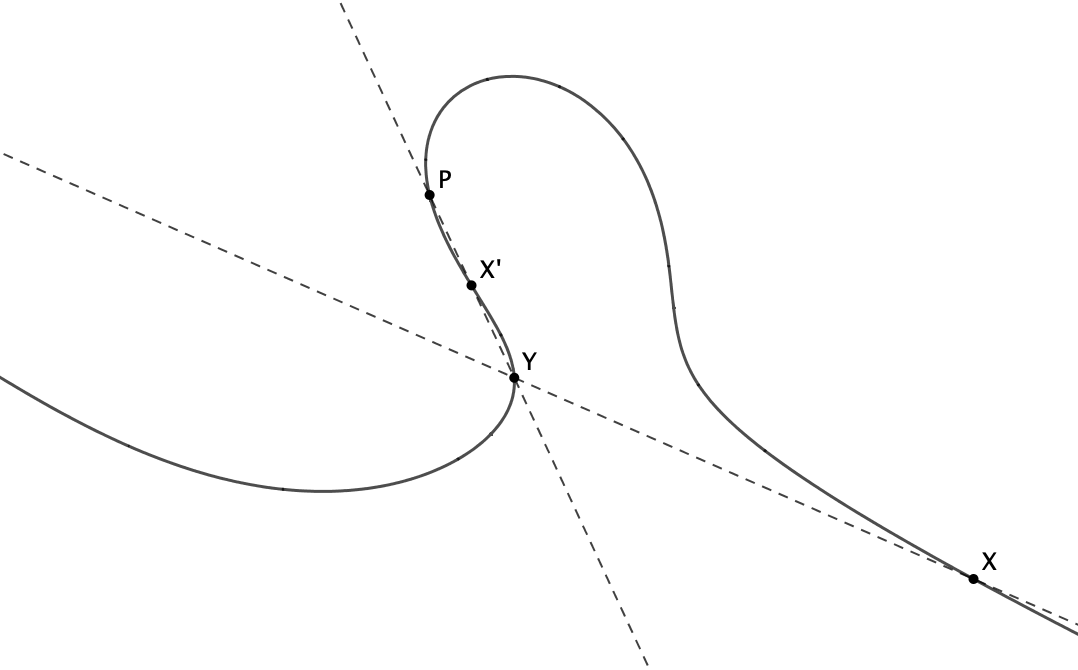}
    \caption{Construction of the Isogonal Conjugatate in a Cubic}
    \label{fig:6.1}
\end{figure}

Therefore, given any non-degenerate isogonal cubic $\mathcal C \in \mathbb R^2$ and any point $X \in \mathcal C$, the spiral center $P$ and the isogonal conjugate of $X$ with respect to $\mathcal C$ are well-defined. Thus, we may now revisit our constructions of intersections and tangents, this time with a general isogonal cubic.

\begin{theorem}[Tangents to the Isogonal Cubic]
    For non-singular $X \in \mathcal C$, let $X'$ be its isogonal conjugate. Let $\ell$ be the isogonal of $XX'$ wrt lines $XP, XP_\infty$. Then $\ell$ is tangent to $\mathcal C$ at $X$.
\end{theorem}

\begin{theorem}[Line Intersections in the Isogonal Cubic]
    For distinct $X, Y \in \mathcal C$, let $\ell_X$ be the isogonal of $XY$ wrt lines $XP, XP_\infty$; define $\ell_Y$ analogously. Then $XY \cap \mathcal C$ is the isogonal conjugate of $\ell_X \cap \ell_Y$.
\end{theorem}

\section{Algebraic Characterization in the Cartesian Plane}

To conclude the paper, we present a purely algebraic characterization of all possible isogonal cubics in $\mathbb R^2$ for the sake of completion.

\begin{theorem}
    A non-degenerate cubic $\mathcal C \in \mathbb R^2$ is an isogonal cubic of some quadrilateral $ABCD$ if and only if it has the form $f(x, y) = f(p, q)$, where
    \[f(x, y) = Ax^3 + Bx^2y + Axy^2 + By^3 + Cx^2 + Dxy + Ey^2 + Fx + Gy\]
    such that all coefficients are real and $(A, B) \ne (0, 0)$, and
    \[p = \frac{AE-AC-BD}{2(A^2+B^2)}, \quad\quad q = \frac{BC-AD-BE}{2(A^2+B^2)}.\]
    Furthermore, the spiral center of $\mathcal C$ is $(p, q)$, and the unique real asymptote of $\mathcal C$ is given by
    \[(A^3+AB^2)x + (A^2B+B^3)y + (A^2E-ABD+B^2C) = 0.\]
\end{theorem}

\begin{proof}
    Let the embedding $\mathcal C_0$ of $\mathcal C$ in $\mathbb{CP}^2$ have equation $g(x, y, z) = 0$, where
    \[g(x, y, z) = Ax^3 + Bx^2y + Axy^2 + By^3 + Cx^2z + Dxyz + Ey^2z + Fxz^2 + Gyz^2 + Hz^3\]
    where the equality of the coefficients of $x^3$ with $xy^2$ and $x^2y$ with $y^3$ is given by part (b) of \Cref{class-real}. Let $g$ denote the left-hand side of the above equation. We compute
    \[\frac{\partial g}{\partial x} = 3Ax^2 + 2Bxy + 2Fxz + Ay^2 + Dyz + Fz^2\]
    \[\frac{\partial g}{\partial y} = 3By^2 + 2Axy + 2Eyz + Bx^2 + Dxz + Gz^2\]
    \[\frac{\partial g}{\partial z} = 3Hz^2 + 2Fxz + 2Gyz + Cx^2 + Dxy + Ey^2\]
    Plugging in the partial derivatives for $(1: i: 0)$, the tangent to $\mathcal C_0$ at $(1: i: 0)$ has equation
    \[(2A+2Bi)x + (-2B+2Ai)y + (C+Di-E)z = 0\]
    and similarly the tangent to $\mathcal C_0$ at $(1: -i: 0)$ has equation
    \[(2A-2Bi)x + (-2B-2Ai)y + (C-Di-E)z = 0\]
    The spiral center $P$ of $\mathcal C$ is then given by the solution to these two equations. Solving yields
    \[(x: y: z) \quad = \quad \left(AE-AC-BD: BC-AD-BE: 2(A^2+B^2)\right)\]
    Since $A$ and $B$ are not both $0$, converting back to Cartesian coordinates implies that $P$ indeed has coordinates given by $(p, q)$. In Cartesian coordinates, the value
    \[Ax^3 + Bx^2y + Axy^2 + By^3 + Cx^2 + Dxy + Ey^2 + Fx + Gy\]
    must be a constant, particularly $-H$. Plugging in $(p, q)$ immediately gives the equation for $\mathcal C$ to be $f(x, y) = f(p, q)$ as desired.
    
    To determine the asymptote, we find that points of infinity on $\mathcal C_0$ are given by
    \[0 = Ax^3 + Bx^2y + Axy^2 + By^3 = (x^2+y^2)(Ax+By)\]
    so the real point of infinity is given by $P_\infty = (B: -A: 0)$. Plugging this into the equations for the partial derivatives yields that the tangent to $\mathcal C_0$ at $P_\infty$, and by extension the unique real asymptote of $\mathcal C$, indeed takes the above equation. This completes the proof.
\end{proof}

\section{Acknowledgements}

The author would like to acknowledge and thank to Anant Mudgal for providing some of the necessary theory in complex projective geometry, helping to formalize the analytic characterizations, and checking over the solutions. The author would also like to give special thanks to Michael Diao for checking the paper, suggesting additions to the content, and helping to format the document and diagrams.


\begin{thebibliography}{9}
\bibitem{ref:EGMO} Chen, Evan. \textit{Euclidean Geometry in Mathematical Olympiads}. Washington, DC: Mathematical Association of America, 2016.
\bibitem{ref:Three} Chen, Evan. ``Three Properties of Isogonal Conjugates." Power Overwhelming, Wordpress, 30 November 2014, \url{https://usamo.wordpress.com/2014/11/30/three-properties-of-isogonal-conjugates/}.
\bibitem{ref:Fulton} Fulton, William. \textit{Algebraic Curves}. Addison-Wesley, 1989, pp. 37-63.
\bibitem{ref:Elementary} Lehmer, Derrick Norman. \textit{Elementary Course in Synthetic Projective Geometry}. General Books, 2010, pp. 88-103.
\bibitem{ref:Cayley} Weisstein, Eric W. ``Cayley-Bacharach Theorem." Wolfram MathWorld, Wolfram Research, Inc., 25 November 2019, \url{http://mathworld.wolfram.com/Cayley-BacharachTheorem.html}.
\bibitem{ref:Circle} Weisstein, Eric W. ``Circular Point at Infinity." Wolfram MathWorld, Wolfram Research, Inc., 25 November 2019, \url{http://mathworld.wolfram.com/CircularPointatInfinity.html}.
\bibitem{ref:Euler} Weisstein, Eric W. ``Euler's Homogeneous Function Theorem." Wolfram MathWorld, Wolfram Research, Inc., 25 November 2019, \url{http://mathworld.wolfram.com/EulersHomogeneousFunctionTheorem.html}.
\bibitem{ref:Gauss} Weisstein, Eric W. ``Gauss." Wolfram MathWorld, Wolfram Research, Inc., 25 November 2019, \url{http://mathworld.wolfram.com/Gauss-BodenmillerTheorem.html}.
\bibitem{ref:Inconic} Weisstein, Eric W. ``Inconic." Wolfram MathWorld, Wolfram Research, Inc., 25 November 2019, \url{http://mathworld.wolfram.com/Inconic.html}.
\end{thebibliography}
\end{document}